\documentclass{amsart}
\usepackage[english]{babel}
\usepackage[latin1]{inputenc}
\usepackage[dvips,final]{graphics}
\usepackage{amsmath,amsfonts,amssymb,amsthm,amscd,array,
stmaryrd,mathrsfs,mathdots,epigraph}
\usepackage[makeroom]{cancel}
\usepackage{pstricks}
\usepackage[all]{xy}
\usepackage{url}
\usepackage{multirow, blkarray}
\usepackage{booktabs}
\usepackage{caption}
\usepackage{subcaption}

\usepackage{blindtext}
\usepackage{textcomp}
\usepackage[final]{epsfig}
\usepackage{color}
\usepackage{mathabx}

\theoremstyle{plain}
\newtheorem{thm}{Theorem}
\newtheorem{lem}[thm]{Lemma}
\newtheorem{prop}[thm]{Proposition}

\theoremstyle{definition}

\newtheorem*{rem}{Remark}

\newcommand{\Z}{\mathbb{Z}}

\def\a{\alpha}
\def\b{\beta}

\def\e{\varepsilon}
\def\g{\gamma}

\def\s{\sigma}

\begin{document}

\title{A shadow Markov equation}

\author{Nathan Bonin}
\address{
Nathan Bonin,
Laboratoire de Math\'ematiques de Reims, UMR9008 CNRS,
Universit\'e de Reims Champagne-Ardenne,
U.F.R. Sciences Exactes et Naturelles,
Moulin de la Housse - BP 1039,
51687 Reims cedex 2,
France}
\email{nathan.bonin1@etudiant.univ-reims.fr}

\author{Valentin Ovsienko}
\address{
Valentin Ovsienko,
CNRS,
Laboratoire de Math\'ematiques de Reims, UMR9008 CNRS,
Universit\'e de Reims Champagne-Ardenne,
U.F.R. Sciences Exactes et Naturelles,
Moulin de la Housse - BP 1039,
51687 Reims cedex 2,
France}
\email{valentin.ovsienko@univ-reims.fr}

\begin{abstract}
We introduce an analogue of the classical Markov equation that involves dual numbers $a+\a\e$ with $\e^2=0$.
This equation characterizes the ``shadow Markov numbers'' recently considered by one of us.
We show that this equation is characterized by invariance by cluster algebra mutations.
\end{abstract}

\maketitle

\thispagestyle{empty}

\section{Introduction}

The classical Markov (or Markoff) equation~\cite{Mar} is the Diophantine equation
$$
a^2+b^2+c^2=3abc.
$$
A positive integer $a$ is called a {\it Markov number} if it can be completed to a triplet
of positive integers $(a,b,c)$ satisfying the Markov equation.
Markov numbers attracted much interest in many branches of mathematics,
from number theory to topology, combinatorics, and mathematical physics.

In this note we will consider pairs of positive integers, $(a,\a)\in\Z\times\Z$, that will be written in a form
\begin{equation}
\label{DualEq}
A=a+\a\e,
\end{equation}
where $\e$ is a formal parameter such that $\e^2=0$.
Formal sums~\eqref{DualEq} are called {\it dual numbers}.
The notion of dual number goes back to Clifford and is used in
various situations in algebra, geometry, and mathematical physics.
A dual number~\eqref{DualEq} is called a {\it positive integer} if $a,\a\in\Z_{>0}$.

Our goal is to introduce and study a deformation of the Markov equation
\begin{equation}
\label{SMarEq}
A^2+B^2+C^2=\left(3-2\e\right)ABC.
\end{equation}
Solutions $(A,B,C)$ to this equation are triples of dual numbers.
We call a {\it dual Markov number} a positive integer dual number~$A$  that can be completed to a triplet $(A,B,C)$ 
of positive integer dual numbers satisfying~\eqref{SMarEq}.

The following triplets of dual numbers
\begin{eqnarray*}
(1,& 1+\e, & 1+\e),\\
(1+\e, & 1+\e,& 2+4\e),\\
( 1+\e,& 2+4\e,& 5+13\e)
\end{eqnarray*}
are examples of triplets of dual Markov numbers.

We study several properties of the deformed Markov equation~\eqref{SMarEq}.
The main property is that it is preserved by cluster algebra mutations, similarly to the classical case~\cite{Pro}.
We will show that it is characterized by this mutation invariance.

The notion of ``shadow'' sequences of integers appeared in~\cite{Ovs}, and was tested on
the sequence of Markov numbers.
Every Markov number, $a$, is accompanied by another positive integer, $\a$, 
called the shadow of $a$.
For a definition, see~\cite{Ovs} and Section~\ref{MSec}.
It is {\it a priori} not clear, and was asked in~\cite{Ovs}, if the sequence of shadows of Markov numbers satisfy any equation.
It turns out that the classical Markov numbers organized together with their shadows in the form of dual numbers
are solutions to~\eqref{SMarEq}.

\begin{thm}
\label{MainThm}
Given a Markov number $a$ and its shadow $\a$, the corresponding dual number $A=a+\a\e$ is a dual Markov number.
\end{thm}

We do not have a classification of dual Markov numbers.
A classical theorem Markov implies that every solution $(A,B,C)$ to~\eqref{SMarEq} can be obtained
by a sequence of transformations called {\it mutations} (see Section~\ref{MSec}) from a triplet of the form
$$
A'=1+\a'\e,\qquad
B'= 1+\b'\e,\qquad
C'=1+\g'\e,
$$
with some $\a',\b',\g'\in\Z$, but it is an open problem to determine the conditions that guarantee positivity of $\a,\b,\g$.

Another open problem is to understand whether the equation~\eqref{SMarEq} is connected to supergeometry.
We believe that~\eqref{SMarEq} should be considered as a ``superanalogue'' of the classical Markov equation
and is related to the supergroup $\mathrm{OSp}(1|2,\Z)$; see~\cite{CO}.
However, we were unable to prove this so far.
One also notices a certain resemblance of~\eqref{SMarEq} 
with the super-Markov equation of~\cite{HPZ}; see Section~\ref{InMutS}.

Although the proofs of the statements of this note are rather elementary,
we hope that equation~\eqref{SMarEq} can be useful for what is sometimes called
in mathematical physics ``super number theory''; see~\cite{Rab,CO}.
The idea of shadow integer sequences recently initiated in~\cite{OT} and~\cite{Ovs}
was further developed in~\cite{Hon,Ves,CO} and generated various unexpected connections.

\section{Markov numbers and their shadows}\label{MSec}

In this section we briefly explain the connection of Markov numbers
with Fomin-Zelevinsky cluster algebras~\cite{FZ1},
that was proposed in~\cite{Pro}.
We also recall the construction of the shadow Markov sequence
introduced in~\cite{Ovs}.

\subsection{Markov numbers and quiver mutations}

It is well-known and due to Markov~\cite{Mar}, that every positive integer solution to the Markov equation
can be obtained from $(1,1,1)$ by a sequence of transformations
$(a,b,c)\mapsto(a',b,c)$, where
\begin{equation}
\label{Marmut}
a'=\frac{b^2+c^2}{a},
\end{equation}
and permutations of $a,b,$ and $c$.
It is easy to check that $(a',b,c)$ remains a solution if $(a,b,c)$ is a solution.

It was noticed in~\cite{Pro}, that the maps~\eqref{Marmut} can be understood as instances of {\it cluster algebra mutations}
 in the sense of Fomin-Zelevinsky cluster algebra~\cite{FZ1}.
The quiver associated with the sequence of Markov numbers is as follows:
$$
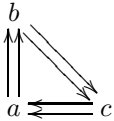
\xymatrix{
b\ar@<2pt>@{->}[rd]\ar@<-2pt>@{->}[rd]\ar@<-2pt>@{<-}[d]\ar@<2pt>@{<-}[d]&\\
a&c\ar@<-2pt>@{->}[l]\ar@<2pt>@{->}[l]
}
$$
and~\eqref{Marmut} are precisely the exchange relations for the  mutations associated to this quiver.

Integrality of $a'$ in~\eqref{Marmut} follows from the Laurent phenomemon~\cite{FZ},
and can also be obtained directly from an equivalent 
(and more frequently used) form of~\eqref{Marmut}:
$a'=3bc-a$.

\subsection{The construction of~\cite{Ovs}}

The notion of shadow Markov sequence~\cite{Ovs} is based on transformations~\eqref{Marmut}.
The shadow sequence is constructed recurrently without any equation on it.

Choose the initial triplet of dual numbers $(A_1,B_1,C_1)$ with
\begin{equation}
\label{InitD}
A_1=1+\a_1\e,\qquad
B_1= 1+\b_1\e,\qquad
C_1=1+\g_1\e,
\end{equation}
where $\a_1,\b_1,\g_1$ are arbitrary integers.
Define an infinite sequence of transformations $(A,B,C)\mapsto(A',B,C)$
by the same formula as~\eqref{Marmut}:
\begin{equation}
\label{SMarmut}
A':=\frac{B^2+C^2}{A}.
\end{equation}
More explicitly~\eqref{SMarmut} reads
\begin{equation}
\label{SMarmutBis}
a'=\frac{b^2+c^2}{a},
\qquad\qquad
\a'=\frac{2b\,\b+2c\,\g-a'\a}{a}
\end{equation}
and therefore determines both, the real and nilpotent parts of $A'=a'+\a'\e$.
Furthermore, $\a'$ will remain integer, which follows from the version of
Laurent phenomenon proved in~\cite{OZ}.

\subsection{The Markov tree and its shadow}
Similarly to the classical Markov numbers,
their shadows can be organized with the help of an infinite binary tree.
The version that we use is the following.

The binary tree is drawn in the plane cutting it into infinitely many regions,
and every region is labeled by a Markov number.
Locally the picture is this
$$
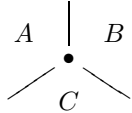
\xymatrix @!0 @R=0.3cm @C=0.3cm
{
&&&\ar@{-}[ddd]\\
\\
&A&&&&B\\
&&&\bullet\ar@{-}[llldd]\ar@{-}[rrrdd]\\
\\
&&&C&&&
}
$$
and the transformations~\eqref{SMarmut} correspond to the following branchings:
$$
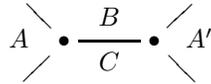
\xymatrix @!0 @R=0.3cm @C=0.3cm
{
\ar@{-}[rrdd]&&&&&&&&\\
&&&&B\\
A&&\bullet\ar@{-}[lldd]\ar@{-}[rrrr]&&&&
\bullet\ar@{-}[rruu]\ar@{-}[rrdd]&&A'\\
&&&&C\\
&&&&&&&&
}
$$

The most natural choice of the initial values $(\a_1,\b_1,\g_1)$ is
\begin{equation}
\label{InitGooD}
(\a_1,\b_1,\g_1)=(0,1,1).
\end{equation}
For details; see~\cite{Ovs} and Section~\ref{ICSec}.
The tree of Markov numbers with the initial triplet~\eqref{InitGooD} is presented below together with their shadow:
$$
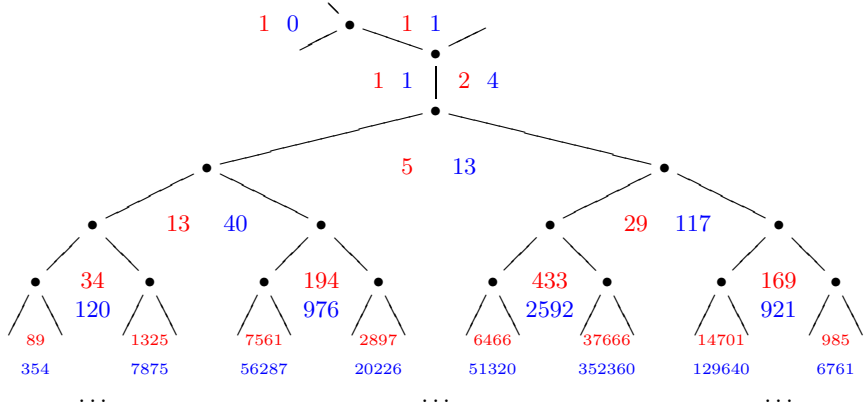
\begin{small}
\xymatrix @!0 @R=0.38cm @C=0.38cm
{&&&&&&&&&&&&\\
&&&&&&&&&&{\textcolor{red}{1}}&{\textcolor{blue}{0}}
&&\bullet\ar@{-}[lld]\ar@{-}[lu]&&{\textcolor{red}{1}}&{\textcolor{blue}{1}}&&\\
&&&&&&&&&&&&&&&&\bullet\ar@{-}[dd]\ar@{-}[rru]\ar@{-}[lllu]\\
&&&&&&&&&&&&&&{\textcolor{red}{1}}&{\textcolor{blue}{1}}&&{\textcolor{red}{2}}&{\textcolor{blue}{4}}\\
&&&&&&&&&&&&&&&&\bullet\ar@{-}[lllllllldd]\ar@{-}[rrrrrrrrdd]&&&&&&&&\\
&&&&&&&&&&&&&&&&\\
&&&&&&&&\bullet\ar@{-}[lllldd]\ar@{-}[rrrrdd]
&&&&&&&{\textcolor{red}{5}}&&{\textcolor{blue}{13}}
&&&&&&&\bullet\ar@{-}[lllldd]\ar@{-}[rrrrdd]&&&\\
&&&&&&&&
&&&&&&&&&&&&&&&&\\
&&&&\bullet\ar@{-}[lldd]\ar@{-}[rrdd]
&&&{\textcolor{red}{13}}&&{\textcolor{blue}{40}}&&&\bullet\ar@{-}[lldd]\ar@{-}[rrdd]
&&&&&&&&\bullet\ar@{-}[lldd]\ar@{-}[rrdd]
&&&{\textcolor{red}{29}}&&{\textcolor{blue}{117}}&&&\bullet\ar@{-}[lldd]\ar@{-}[rrdd]\\
&&&&&&&&&&&&&&&&&&&&&&&&&&&&\\
&&\bullet\ar@{-}[ldd]\ar@{-}[rdd]
&&{\textcolor{red}{34}}&&\bullet\ar@{-}[ldd]\ar@{-}[rdd]
&&&&\bullet\ar@{-}[ldd]\ar@{-}[rdd]
&&{\textcolor{red}{194}}&&\bullet\ar@{-}[ldd]\ar@{-}[rdd]
&&&&\bullet\ar@{-}[ldd]\ar@{-}[rdd]
&&{\textcolor{red}{433}}&&\bullet\ar@{-}[ldd]\ar@{-}[rdd]
&&&&\bullet\ar@{-}[ldd]\ar@{-}[rdd]
&&{\textcolor{red}{169}}&&\bullet\ar@{-}[ldd]\ar@{-}[rdd]\\
&&&&{\textcolor{blue}{120}}&&&&&&&&{\textcolor{blue}{976}}&&&&&&&&{\textcolor{blue}{2592}}&&
&&&&&&{\textcolor{blue}{921}}&&&&&\\
&&{\textcolor{red}{\scriptstyle89}}
&&&&{\textcolor{red}{\scriptstyle1325}}
&&&&{\textcolor{red}{\scriptstyle7561}}
&&&&{\textcolor{red}{\scriptstyle2897}}
&&&&{\textcolor{red}{\scriptstyle6466}}
&&&&{\textcolor{red}{\scriptstyle37666}}
&&&&{\textcolor{red}{\scriptstyle14701}}
&&&&{\textcolor{red}{\scriptstyle985}}&&\\
&&{\textcolor{blue}{\scriptstyle354}}
&&&&{\textcolor{blue}{\scriptstyle7875}}
&&&&{\textcolor{blue}{\scriptstyle56287}}
&&&&{\textcolor{blue}{\scriptstyle20226}}
&&&&{\textcolor{blue}{\scriptstyle51320}}
&&&&{\textcolor{blue}{\scriptstyle352360}}
&&&&{\textcolor{blue}{\scriptstyle129640}}
&&&&{\textcolor{blue}{\scriptstyle6761}}
\\
&&&&\ldots&&&&&&&&&&&&\ldots&&&&&&&&&&&&\ldots
}
\end{small}$$

\begin{rem}
The simplest subsequence of Markov numbers (left branch of the above tree) 
contains the odd Fibonacci numbers $1,2,5,13,34,89,\ldots$.
It was noticed in~\cite{Ovs}, that the sequence 
$$
1,4,13,40,120,354,1031,2972, 8495,\ldots
$$
appearing as the shadow of the odd Fibonacci branch,
is the sequencel A238846 (see~\cite{OEIS}) which is the convolution of two
bisections of the Fibonacci sequence, $F_{2n+1}$ and $F_{2n}$.
We refer to~\cite{MOZ,Ovs} for various $\e$-deformations of Fibonacci numbers.

The combinatorial nature of other shadow sequence appearing in the above tree
is not understood.
For instance, it would be interesting to know properties of the sequence
$1,13,117,921,6761,\ldots$
appearing as the shadow of
$1, 5, 29,169, 985,\ldots$ which is a bisection of the classical Pell numbers.
\end{rem}

\section{Mutation stability}

In this section, we relax the positivity condition for the nilpotent part
$\a$ of dual numbers $A=a+\a\e$.
We consider an arbitrary choice of the initial triplet~\eqref{InitD}
and give a slightly more general form of the deformed Markov equation.
We also obtain several properties of this equation.

\subsection{Proof of Theorem~\ref{MainThm}}

The following statement is a more precise and general reformulation of Theorem~\ref{MainThm}.

\begin{prop}
\label{Refor}
(i)
Every triplet of dual numbers $(A,B,C)$ obtained from~\eqref{InitD} by
a sequence of transformations~\eqref{SMarmut} satisfies the equation
\begin{equation}
\label{GenE}
A^2+B^2+C^2=\left(3-\s\e\right)ABC,
\end{equation}
where $\s=\a_1+\b_1+\g_1$.

(ii)
Conversely, every triplet of solutions $(A,B,C)$ to~\eqref{GenE} with positive $a,b,c$
can be obtained from a triple of the form~\eqref{InitD} by a sequence of transformations~\eqref{SMarmut}
mixed with permutations of $A,B$ and~$C$.
\end{prop}

\begin{proof}
To prove Part (i),
let us first check that the initial triplet~\eqref{InitD} satisfies~\eqref{GenE}.
Indeed, substituting~\eqref{InitD} to the left-hand-side of the equation gives
$3+2(\a_1+\b_1+\g_1)\e$, while in the right-hand-side one obtains
$3+3(\a_1+\b_1+\g_1)\e-\s\e$.

Let us check that~\eqref{GenE} is stable with respect to the transformations~\eqref{SMarmut}.

\begin{lem}
\label{StabLem}
The equation
\begin{equation}
\label{GenEX}
A^2+B^2+C^2=X\,ABC,
\end{equation}
where~$X$ is a formal parameter, is preserved by the transformations~\eqref{SMarmut}.
\end{lem}

\noindent{\it Proof of the lemma}.
Let us show that, if $(A,B,C)$ is a solution to~\eqref{GenEX},
then $(A',B,C)$, with $A'$ as in~\eqref{SMarmut}, is also a solution to the same equation.

Assume that $(A,B,C)$ is a solution to~\eqref{GenEX}.
Then $A'$ given by~\eqref{SMarmut} can be rewritten:
\begin{equation}
\label{A'X}
A'=X\,BC - A.
\end{equation}
Substituting this expression for $A'$ to the left-hand-side of~\eqref{GenEX}, one obtains
\begin{eqnarray*}
{A'}^2+B^2+C^2&=&
X^2\,B^2C^2 - 2X\,ABC+A^2+B^2+C^2\\
&=&
X^2\,B^2C^2 - X\,ABC,
\end{eqnarray*}
while in the right-hand-side, one has
\begin{eqnarray*}
X\,A'BC &=& X\,\left(X\,BC-A\right)BC\\
&=&
X^2\,B^2C^2 - X\,ABC.
\end{eqnarray*}
We conclude that the triplet $(A',B,C)$ satisfies exactly the same equation, namely
${A'}^2+B^2+C^2=X\,A'BC$.
\qed

\medskip

The assumption that the initial triplet~$(A_1,B_1,C_1)$ satisfies~\eqref{GenEX} implies that
$$
X=\frac{A_1^2+B_1^2+C_1^2}{A_1B_1C_1},
$$
that gives $X=3-(\a_1+\b_1+\g_1)\e$.

Proposition~\ref{Refor}, Part (i) follows.

To prove Part (ii), we use the classical theorem that any positive solution $(a,b,c)$ to the Markov equation
can be reduced to $(1,1,1)$ by a sequence of transformations~\eqref{Marmut}.
We conclude following the same sequence of transformations~\eqref{SMarmut}.
\end{proof}

\subsection{Uniqueness}\label{InMutS}

Let us show that~\eqref{GenE} is the only $\e$-deformation of the Markov equation
which is preserved by the mutations~\eqref{SMarmut}.

Consider the most general $\e$-deformed Markov equation, it can be written in the form
\begin{equation}
\label{DefEq}
A^2+B^2+C^2=3ABC+P(A,B,C)\e,
\end{equation}
where $P$ is an arbitrary polynomial
$$
P(A,B,C)=\sum_{i,j,k\geq0}P_{ijk}A^iB^jC^k.
$$

\begin{thm}
\label{InvThm}
Equation~\eqref{GenE} is the only equation among the family~\eqref{DefEq}
preserved by the transformations~\eqref{SMarmut}.
\end{thm}

\begin{proof}
Let us show that,
if~\eqref{DefEq} is preserved by the transformations~\eqref{SMarmut}, then all the coefficients
$P_{ijk}$ of the polynomial~$P$ vanish, except perhaps~$P_{111}$, which is arbitrary.
Indeed,~\eqref{SMarmut} can be rewritten
$$
A'=3BC-A+\frac{P(A,B,C)}{A}\e
$$
so that the left-hand-side of~\eqref{DefEq} after the transformation becomes
$$
\begin{array}{rcl}
{A'}^2+B^2+C^2&=&
9B^2C^2 - 6ABC+A^2+B^2+C^2
\\[4pt]
&&+\left(\frac{6BC}{A}P(A,B,C)-2P(A,B,C)\right)\e\\[4pt]
&&=9B^2C^2 - 3ABC+\left(\frac{6BC}{A}P(A,B,C)-P(A,B,C)\right)\e.
\end{array}
$$
The right-hand-side of~\eqref{DefEq} is
$$
3A'BC +P(A',B,C)\e=9B^2C^2-3ABC+\left(\frac{3BC}{A}P(A,B,C)+P(A',B,C)\right)\e.
$$
The assumption the equation is preserved by the mutations then reads
$$
P(A',B,C)\e=\left(\frac{3BC}{A}-1\right)P(A,B,C)\e,
$$
or with more details
$$
\sum_{i,j,k\geq0}P_{ijk}(3BC-A)^iB^jC^k=
\left(\frac{3BC}{A}-1\right)\sum_{i,j,k\geq0}P_{ijk}A^iB^jC^k.
$$
(Note that these terms appear with multiple~$\e$, so that we omit 
$\frac{P(A,B,C)}{A}\e$ in the left-hand-side since $\e^2=0$.)

The latter equation implies that $P_{ijk}=0$ for $i>1$, and similarly for $j$ and~$k>1$.
Indeed, taking the maximal~$i$, for which~$P_{ijk}\ne0$, if $i\geq2$,
all of the terms in the right-hand-side contain~$A$, while the left-hand-side contain the term
$P_{ijk}3^iB^{i+j}C^{i+k}$.
Therefore, the equation~\eqref{DefEq} cannot be invariant in this case.

The theorem follows.
\end{proof}

\begin{rem}
Another version of super-Markov equation was suggested
by Huang, Penner, and Zeitlin~\cite{HPZ}:
$$
a^2 +b^2 +c^2 +\left(ab+bc+ac\right)\e=3\left(1+\e\right)abc.
$$
For a discussion and a related conjecture; see~\cite{MOZ}.
This equation has interesting geometric properties, albeit it is not preserved by the mutations~\eqref{Marmut}.
\end{rem}

\section{Initial conditions}\label{ICSec}

Let us end the note with a brief discussion of the choice of initial triplet $\a_1,\b_1,\g_1$ in~\eqref{InitD}.
The choice of~\cite{Ovs}, which is $(\a_1,\b_1,\g_1)=(0,1,1)$ is not unique,
but can be justified by the following.

A simple property of the equation~\eqref{GenE} is
that the shadow parts of its solutions linearly depend on the
choice of the initial triplet $\a_1,\b_1,\g_1$.
This readily follows from the second equation in~\eqref{SMarmutBis}.
Therefore, the shadow part of every solution to~\eqref{GenE} is a linear combination of those
with one of the following initial triplets
$$
(0,1,1),
\qquad
(1,1,1),
\qquad
(0,1,0),
$$
that form a basis in the space of initial conditions.
However, choosing $(1,1,1)$, one obtains the shadow part that coincides with the Markov numbers,
namely the solutions of the form $(A,B,C)=(a+a\e,b+b\e,c+c\e)$, which is not an interesting case.
The choice $(0,1,0)$ produces negative integers.
As mentioned in the introduction,
it would be interesting to characterize the initial triplets that produce only positive shadow parts.

\medskip
 \noindent
{\bf Acknowledgements}.
We would like to thank
Philippe Larchev\^eque, Sophie Morier-Genoud, Julien Rouyer, Michael Shapiro, Sergei Tabachnikov, 
Alexey Ustinov, and Alexander Veselov for fruitful discussions.
We are indebted to the anonymous referee for many helpful remarks and suggestions.
This project was partially supported by the ANR project ANR-19-CE40-0021.

\end{document}